\documentclass{amsart}

\usepackage[defaultlines=3,all]{nowidow}
\usepackage[section]{placeins}
\usepackage{ pbox }         % allows creation of flexible parboxes

\usepackage{ mathtools, amssymb, stmaryrd, mathdots, bbm, mleftright, faktor }
\usepackage{etoolbox, xparse}
\usepackage{suffix}         % allows defining of starred commands, by putting "\WithSuffix" before "\newcommand"

\usepackage[shortlabels]{ enumitem }
\usepackage{ url }

\newcommand{\initialsspace}{0.1em}

\makeatletter
\newcommand{\splitlist}[1]{\@splitlist#1\@nil}
\def\@splitlist#1\@nil{%
  \if\relax\detokenize{#1}\relax
    \expandafter\@gobble
  \else
    \expandafter\@firstofone
  \fi
  {\@spl@tlist#1.\@nil}%
}
 
\def\@spl@tlist#1.#2\@nil{%
    \def\tmpA{#1}%
    \def\tmpB{#2}%
    \def\tmpP{.}%
    \ifx\tmpB\tmpP%
        #1.%
    \else{%
        \ifx\tmpA\@empty%  
        \else%
                #1.\nobreak\hspace{\initialsspace}%
        \fi%
    }% 
    \fi%
  \if\relax\detokenize{#2}\relax
    \expandafter\@firstoftwo
  \else
    \expandafter\@secondoftwo
  \fi
  {\unskip}%
  {\@spl@tlist#2\@nil}%
}
\makeatother
 
\newcommand{\initials}[1]{\splitlist{#1}}

\usepackage{ thmtools }

\usepackage[hyphenbreaks]{breakurl} % use before hyperref so that links are always able to be broken over a line (arXiv does not allow this by default)
\usepackage{xcolor}
\definecolor{dark-blue}{rgb}{0.15,0.15,0.4}
\usepackage[colorlinks=true, linkcolor=black, citecolor=black, urlcolor=dark-blue, anchorcolor=dark-blue]{ hyperref }

\usepackage[nameinlink]{ cleveref }

\numberwithin{equation}{section}
\declaretheorem[style=plain,numberlike=equation]{theorem}
\declaretheorem[style=plain,numberlike=theorem]{lemma}
\declaretheorem[style=plain,numberlike=theorem]{proposition}
\declaretheorem[style=plain,numberlike=theorem]{corollary}
\declaretheorem[style=remark,numberlike=theorem]{remark}
\declaretheorem[style=definition,numberlike=theorem]{definition}

\declaretheorem[name=Example,style=definition,numbered=no]{example*}

\makeatletter
\declaretheoremstyle[
    spaceabove=6pt,
    spacebelow=6pt,
    notebraces={}{},
    notefont=\bfseries,
    headformat=\let\thmt@space\@empty\NOTE,
    bodyfont=\itshape,
]{namedStyle}
\makeatother

\makeatletter    
\newcommand{\prelistcommand}{\nobreak\leavevmode\@nobreaktrue}
\makeatother

\SetEnumitemKey{beginthm}{before=\prelistcommand}   % use this key when the list environment begins a theorem statement, to remove the excess vertical spacing between the theorem header and the list
\SetEnumitemKey{smallin}{leftmargin=2.7em, labelindent=*}   % use this key to decrease the indentation of the list environment

\setlist{smallin, topsep=1pt}

\crefformat{section}{\S#2#1#3}
\crefrangeformat{section}{\S#3#1#4--#5#2#6}
\crefmultiformat{section}{\S#2#1#3}{ and~\S#2#1#3}{, \S#2#1#3}{, and~\S#2#1#3}
\crefrangemultiformat{section}{\S#2#1#3}{ and~\S#2#1#3}{, \S#2#1#3}{, and~\S#2#1#3}

\usepackage{ bbm } 
\usepackage{ amsbsy } % allows emboldening of arbitrary characters in mathmode using \pmb{}

\NewDocumentCommand\set{s m}{%
    \IfBooleanTF#1%
    {\left\{ #2 \right\}}%
    {\{#2\}}%
}
\newcommand{\multisetgap}{4mu}
\NewDocumentCommand\multiset{s m}{%
    \IfBooleanTF#1%
    {\left\{\mkern-\multisetgap\left\{ #2 \right\}\mkern-\multisetgap\right\}}%
    {\{\mkern-\multisetgap\{#2\}\mkern-\multisetgap\}}%
}
\NewDocumentCommand\setbuild{s m m}{%
    \IfBooleanTF#1%
    {\ensuremath{\left\{\, #2 \, \middle| \, #3 \,\right\}}}%
    {\ensuremath{\{\, #2 \, \mid \, #3 \,\}}}%
}
\NewDocumentCommand\spangle{s m m m}{%
    \IfBooleanTF#1%
    {\ensuremath{\left\langle\, #2 \, \middle| \, #3 \,\right\rangle_{#4}}}%
    {\ensuremath{\langle\, #2 \, \mid \, #3 \,\rangle_{#4}}}%
}

\DeclarePairedDelimiter{\ceil}{\lceil}{\rceil}
\DeclarePairedDelimiter{\floor}{\lfloor}{\rfloor}
\DeclarePairedDelimiter{\paren}{(}{)}
\DeclarePairedDelimiter{\abs}{\lvert}{\rvert}
\DeclarePairedDelimiter{\sqbra}{[}{]}
\DeclarePairedDelimiter{\angl}{\langle}{\rangle}
\makeatletter
\let\oldceil\ceil
\def\ceil{\@ifstar{\oldceil}{\oldceil*}}
\let\oldfloor\floor
\def\floor{\@ifstar{\oldfloor}{\oldfloor*}}
\let\oldparen\paren
\def\paren{\@ifstar{\oldparen}{\oldparen*}}
\let\oldabs\abs
\def\abs{\@ifstar{\oldabs}{\oldabs*}}
\let\oldsqbra\sqbra
\def\sqbra{\@ifstar{\oldsqbra}{\oldsqbra*}}
\let\oldangl\angl
\def\angl{\@ifstar{\oldangl}{\oldangl*}}
\makeatother

\newcommand{\Z}{\mathbb{Z}}
\newcommand{\Q}{\mathbb{Q}}

\newcommand{\C}{\mathbb{C}}

\DeclareMathOperator{\stab}{stab}
\newcommand{\Ccl}{\mathcal{C}}

\newcommand{\upto}[1]{\mathord{\uparrow}^{#1}} % induction
\newcommand{\downto}[1]{\mathord{\downarrow}_{#1}} % restriction

\DeclareMathOperator{\diag}{diag}

\DeclareMathOperator{\PGL}{PGL}           % Projective linear group

\makeatletter
\let\@@pmod\pmod
\DeclareRobustCommand{\pmod}{\@ifstar\@pmods\@@pmod}
\def\@pmods#1{\mkern4mu({\operator@font mod}\mkern 6mu#1)}
\makeatother

\DeclareMathOperator{\supp}{supp}       % support
\newcommand{\support}[1]{\supp(#1)}       % support

\newcommand{\blank}{{-}}    % Placeholder for arguments in a function

\renewcommand{\epsilon}{\varepsilon}
\renewcommand{\phi}{\varphi}
\renewcommand{\leq}{\leqslant}

\renewcommand{\geq}{\geqslant}
\renewcommand{\emptyset}{\varnothing}

\usepackage{letltxmacro}
\makeatletter
\let\oldr@@t\r@@t
\def\r@@t#1#2{%
\setbox0=\hbox{$\oldr@@t#1{#2\,}$}\dimen0=\ht0
\advance\dimen0-0.2\ht0
\setbox2=\hbox{\vrule height\ht0 depth -\dimen0}%
{\box0\lower0.4pt\box2}}
\LetLtxMacro{\oldsqrt}{\sqrt}
\renewcommand*{\sqrt}[2][\ ]{\oldsqrt[#1]{#2}}
\makeatother

\makeatletter
\def\subsection{\@startsection{subsection}{2}%
  \z@{.6\linespacing \@minus 0.1\linespacing}{.3\linespacing}%
  {\normalfont\bfseries}}
\makeatother

\newcommand{\vsmallskipamount}{3pt plus 0.5pt minus 0.5pt}
\newcommand{\vvsmallskipamount}{0.75pt plus 0.25pt minus 0.5pt}
\makeatletter
\newcommand{\proofpart}[1]{%
  \par
  \addvspace{\vsmallskipamount}%
  \noindent\emph{#1}\par\nobreak
  \addvspace{\vvsmallskipamount}%
  \@afterheading
}
\makeatother

\makeatletter
\newcommand*\wt[1]{\mathpalette\wthelper{#1}}
\newcommand*\wthelper[2]{%
        \hbox{\dimen@\accentfontxheight#1%
                \accentfontxheight#11.2\dimen@ % change the multiplier after #1 before \dimen to adjust. larger multiplier means lower tilde
                $\m@th#1\widetilde{#2}$%
                \accentfontxheight#1\dimen@
        }%
}

\newcommand*\accentfontxheight[1]{%
        \fontdimen5\ifx#1\displaystyle
                \textfont
        \else\ifx#1\textstyle
                \textfont
        \else\ifx#1\scriptstyle
                \scriptfont
        \else
                \scriptscriptfont
        \fi\fi\fi3
}
\makeatother

\newlength{\truelen}
\newcommand{\padbox}[3][c]{%
    \settowidth{\truelen}{\ensuremath{#2}}%
    \ifdim\truelen < #3%
        \makebox[#3][#1]{\ensuremath{#2}}%
    \else%
        \ensuremath{#2}%
    \fi%
}

\newlength{\minlen}
\newcommand{\flexbox}[3][l]{%
    \settowidth{\minlen}{\ensuremath{#3}}%
    \padbox[#1]{#2}{\minlen}%
}

\newcommand{\plus}{{+}}
\newcommand{\minus}{{-}}
\newcommand{\uno}{{(1)}}
\newcommand{\zed}{{(z)}}

\newcommand{\la}{\lambda}

\newcommand{\cover}[1]{\wt{#1}}
\newcommand{\coverCcl}{\cover{\mbox{\(\Ccl\)}}}
\newcommand{\projmap}{\theta}
\newcommand{\negchar}[1]{\langle #1 \rangle}
\newcommand{\ass}{{\mathsf{a}}}

\newcommand{\el}{l}

\title[Characters determined by values on \(\MakeLowercase{\el}'\)-classes]
{Characters and spin characters of alternating and symmetric groups determined by values on \(\el'\)-classes}
\author[Eoghan McDowell]{Eoghan McDowell}
\usepackage[foot]{ amsaddr }

\makeatletter
\def\@setfoot@addresses{%
    \ifx\@affiliation\@empty\else
    \affiliationname\ \@affiliation \@addpunct.%
    \fi
    \ifx\web@page\@empty\else
    
    \webpagename\ \web@page \@addpunct.%
    \fi
}
\def\affiliationname{\textit{Affiliation}:}
\def\affiliation#1{\gdef\@affiliation{#1}}
\let\@affiliation\@empty
\def\webpagename{\textit{Webpage}:}
\def\webpage#1{\gdef\web@page{#1}}
\let\web@page\@empty
\makeatother

\affiliation{Okinawa Institute of Science and Technology}
\email{eoghan.mcdowell@oist.jp}
\webpage{\href{https://eoghanjmcdowell.com}{eoghanjmcdowell.com}}

\makeatletter
\def\@setsubjclass{%
    \vspace{4pt}

    \vspace{-\baselineskip}
    {\itshape\subjclassname.}\enspace\@subjclass\@addpunct.%
}
\makeatother



\makeatletter
\def\@setthanks{%
\vspace{-\baselineskip}\vspace{4pt}
\def\thanks##1{\@par##1\@addpunct.}
\thankses%
}
\def\journalinfo#1{\thanks{#1}}
\makeatother

\journalinfo{
This is the accepted manuscript for an article to be published in Arkiv f\"or Matematik.
}

\begin{document}

\begin{abstract}
This paper identifies all pairs of ordinary irreducible characters of the alternating group which agree on conjugacy classes of elements of order not divisible by a fixed integer \(\el\), for \(\el \neq 3\).
We do likewise for spin characters of the symmetric and alternating groups.
We find that the only such characters are the conjugate or associate pairs labelled by partitions with a certain parameter divisible by \(\el\).
When \(\el\) is prime, this implies that the rows of the \(\el\)-modular decomposition matrix are distinct except for the rows labelled by these pairs.
When \(\el=3\) we exhibit many additional examples of such pairs of characters.
\end{abstract}

\vspace*{14pt}
\vspace{-0.47cm}
\vspace{-0.37cm}

\maketitle

\vspace{-0.1cm}
\vspace{-0.79cm}

\section{Introduction}

Does knowing its values on \(\el'\)-classes suffice to identify an ordinary irreducible character of a group \(G\)?
Here, an \(\el'\)-class means a conjugacy class of elements of order not divisible by a positive integer \(\el \geq 2\).
The answer is certainly ``no'' if \(\el\) is prime and \(G\) is an \(\el\)-group, since the only \(\el'\)-class is the identity;
on the other hand, if \(G\) is a symmetric group and \(\el > 2\) then the answer is ``yes'': Wildon showed that all irreducible characters of the symmetric group are uniquely determined by their values on \(\el'\)-classes when \(\el > 2\) \cite[Corollary~2.1.3]{wildon2008distinctrows}.

In this paper we resolve this question for the alternating group, and for the spin characters of the symmetric and alternating groups, when \(\el \neq 3\).

When \(\el\) is prime, answering this question gives us information about the modular representation theory of \(G\).
Indeed, for ordinary irreducible characters \(\chi\) and \(\psi\) of a group \(G\), is is straightforward to show (see the proof of \Cref{cor:distinct_rows}) that the following are equivalent:
\begin{enumerate}[(i)]
    \item
\(\chi\) and \(\psi\) agree on the \(\el'\)-classes of \(G\);
    \item
the rows of the \(\el\)-modular decomposition matrix of \(G\) labelled by \(\chi\) and \(\psi\) are equal.
\end{enumerate}
This second property says that the \(\el\)-modular reductions of the representations affording \(\chi\) and \(\psi\) have the same multiset of composition factors (that is, the same Brauer character; for an account of decomposition numbers and Brauer characters, see for example \cite[Chapter~10]{webb2016finitegroupreptheory}).
Thus we obtain as a corollary to our main theorems a classification of repeated rows in the decomposition matrices of the alternating group when \(\el \neq 3\), and of the double covers of the symmetric and alternating groups when \(\el \not\in \set{2,3}\).

\subsection{Main theorems}

\newcommand{\spin}{spin }

Our first main theorem concerns the alternating group \(A_n\);
see \Cref{subsection:chars_of_S_n_and_A_n} for the labelling of its characters.
The case of \(l=2\) is due to \cite[Theorem~3.2.1]{wildon2008distinctrows}.

\begin{theorem}
\label{thm:maintheorem:A_n}
Let \(\el \neq 3\).

If \(\el\) is even,
irreducible characters of \(A_n\) are uniquely determined by their values on \(\el'\)-classes.

If \(\el\) is odd,
a conjugate pair of characters of \(A_n\) labelled by a self-conjugate partition with a principal hook length divisible by \(\el\) agree on \(\el'\)-classes; all other irreducible characters of \(A_n\) are uniquely determined by their values on \(\el'\)-classes.
\end{theorem}

Our next main theorems concern \emph{spin characters}.
The projective representation theory of the symmetric and alternating groups \(S_n\) and \(A_n\) is controlled by the linear representation theory of their double covers \(\cover{S}_n\) and \(\cover{A}_n\);
characters of a double cover which are not lifts from the original group are those which send the central element of order \(2\) to \(-I\), and are called \emph{\spin characters} (or \emph{projective characters} or \emph{negative characters}).
Note that there is a choice of two double covers of \(S_n\), but our results hold for either choice; see \Cref{subsection:chars_of_double_covers} for a justification of this, as well as for the explicit presentation we will use and the labelling of its characters.

\begin{theorem}
\label{thm:maintheorem:double_cover_of_S_n}
Let \(\el \neq 3\).
An associate pair of \spin characters of \(\cover{S}_n\) labelled by an odd partition with a part divisible by \(\el\) agree on \(\el'\)-classes;
all other irreducible \spin characters of \(\cover{S}_n\) are uniquely determined by their values on \(\el'\)-classes.
\end{theorem}

\begin{theorem}
\label{thm:maintheorem:double_cover_of_A_n}
Let \(\el \neq 3\).
A conjugate pair of \spin characters of \(\cover{A}_n\) labelled by an even partition with a part divisible by \(\el\) agree on \(\el'\)-classes;
all other irreducible \spin characters of \(\cover{A}_n\) are uniquely determined by their values on \(\el'\)-classes.
\end{theorem}

When \(\el=2\), the unique determination stated in \Cref{thm:maintheorem:double_cover_of_S_n,thm:maintheorem:double_cover_of_A_n} relies on the character being known to be spin: there are spin and non-spin characters which agree on \(2'\)-classes (see \Cref{subsection:spin_equal_to_non-spin} below).
If \(\el\neq 2\), however, no spin character agrees with a non-spin character on \(\el'\)-classes, since they disagree on the central element of order \(2\).
Thus, together with \Cref{thm:maintheorem:A_n} and \cite[Corollary~2.1.3]{wildon2008distinctrows}, these theorems completely classify pairs of irreducible characters of \(\cover{S}_n\) and \(\cover{A}_n\) which agree on \(\el'\)-classes for \(\el \not\in \set{2,3}\). 

Our results have the following interpretation in terms of decomposition matrices.

\begin{corollary}
\label{cor:distinct_rows}
Let \(p\) be prime.
With the exceptions of the repeated rows labelled by the pairs of characters agreeing on \(p'\)-classes described in \Cref{thm:maintheorem:A_n,thm:maintheorem:double_cover_of_S_n,thm:maintheorem:double_cover_of_A_n},
there are no repeated rows in the decomposition matrices of \(A_n\) when \(p \neq 3\), nor in the decomposition matrices of \(\cover{S}_n\) and \(\cover{A}_n\) when \(p \not\in \set{2,3}\). 
\end{corollary}

\begin{proof}
We argue as in \cite[\S3]{wildon2008distinctrows}.
Let \(D\) be the decomposition matrix, \(B\) the Brauer character table, and \(X\) the ordinary character table restricted to \(p'\)-elements (for any finite group).
Note that \(DB = X\) and that \(B\) is invertible.
This implies that equal rows in \(D\) correspond to equal rows in \(X\) (which are precisely characters agreeing on \(p'\)-classes):
if the rows of \(D\) labelled by ordinary characters \(\chi\) and \(\psi\) are equal, and \(g\) is any \(p'\)-element of the group, then
\[
    \chi(g) = \sum_{\phi} D_{\chi,\phi} \phi(g) = \sum_{\phi} D_{\psi,\phi} \phi(g) = \psi(g)
\]
where the sums are over all Brauer characters \(\phi\) of the group; similarly for the converse by writing \(D = XB^{-1}\).
%
% The only remaining concern, for \(\cover{S}_n\) and \(\cover{A}_n\), is whether a row of the decomposition matrix labelled by a spin character can be equal to a row labelled by a non-spin character.
% But spin and non-spin characters differ on the central element of order \(2\), so for \(p \neq 2\) this does not happen.
\end{proof}

In all three of our main theorems, the infinite family of pairs of characters agreeing on \(\el'\)-classes can be identified immediately:
\begin{itemize}
    \item
for a self-conjugate partition \(\la\), the conjugate characters \(\chi^{\lambda^\plus}\) and \(\chi^{\lambda^\minus}\) of \(A_n\) differ only on the split class of cycle type consisting of the principal hook lengths of \(\lambda\); 
    \item
for an odd partition \(\la\) with distinct parts, the \spin character \(\negchar{\la}\) of \(\cover{S}_n\) and its associate differ only on the split class labelled by \(\la\);
    \item
for an even partition \(\la\) with distinct parts, the restriction to \(\cover{A}_n\) of the \spin character \(\negchar{\la}\) decomposes as a sum of two conjugate characters which differ only on the split class labelled by \(\la\).
\end{itemize}
The force of the theorems is that no other pairs of characters agree on the \(\el'\)-classes.

This is false, however, for \(\el=3\).
That is, there are pairs of characters which agree on the \(3'\)-classes -- and hence there are repeated rows in the \(3\)-modular decomposition matrix -- in addition to the conjugate and associate pairs identified above.
For \(A_n\), we identify two infinite families of such pairs (\Cref{thm:3-indistinguishable}), and remark on a striking similarity with known characters of \(A_n\) with equal vanishing sets.
For \(\cover{S}_n\) and \(\cover{A}_n\), we record all such pairs for \(n \leq 14\) (\Cref{prop:n<=14}).

To prove our main theorems, we consider the subalgebra of the centre of the group algebra generated by the \(\el'\)-class sums, and use central characters to determine the genuine characters.
This is the technique used by Wildon in \cite{wildon2008distinctrows} for the symmetric group.
Unlike for the symmetric group, in our cases it is not possible to generate the entire centre; nevertheless we show we can generate enough to distinguish characters up to conjugacy or associates.

This paper is structured as follows. We recall background on the character theory of symmetric and alternating groups and their double covers in \Cref{section:characters_of_A_n_and_covers}.
We prove a sufficient condition for a pair of characters of a normal subgroup to be conjugate in \Cref{section:char_values_and_conjugacy}, and prove the conjugacy class sum generation results required to use this condition in \Cref{section:class_sum_arithmetic}.
We deduce our main theorems in \Cref{section:main_proofs}.
We exhibit characters agreeing on \(3'\)-classes in \Cref{section:3-indistinguishable}.

\subsection{Remarks on the size of a set needed to determine a character}

For each group in consideration, as \(n\) grows, the number of \(\el'\)-classes grows exponentially, but the proportion of \(\el'\)-classes amongst all conjugacy classes decays exponentially.
These facts can be deduced from the numbers and proportions for \(S_n\) (see \cite[(1.36)]{hardyramanujan1918asymptotic} and \cite[Corollary~4.2]{hagis1971partitions}), noting that the numbers of classes and \(\el'\)-classes for the related groups are bounded linearly in terms of those for \(S_n\).
(More precisely, \cite{hagis1971partitions} asymptotically enumerates the partitions with no part repeated \(\el\) or more times; this is equal to the number of partitions with no part divisible by \(\el\) by Glaisher's Theorem (see for example \cite[Chapter~10]{gdjames1978reptheorysymgroups}).
This is an upper bound on the number of \(\el'\)-classes of \(S_n\), attained when \(\el\) is a prime power;
for a lower bound, consider the number of \(q'\)-classes for \(q\) a prime power factor of \(\el\).)

Our theorems therefore say that characters are determined (with the specified exceptions) by their values on a small proportion of classes.
In fact, the proof of \Cref{thm:generation_theorem} requires only \(\el'\)-classes of cycle type having at most four parts greater than \(1\), and so the characters are determined by their values on this restricted set of \(\el'\)-classes, of size \(O(n^4)\).

Nevertheless, we can determine characters by their values on smaller sets of classes if we remove the \(\el'\) restriction.
As noted in \cite{wildon2008distinctrows}, for the symmetric group \(S_n\), the values on cycles -- a set of size \(n\) -- suffice to determine a character: the class sums of the cycles generate the centre of the group algebra \cite{kramer1966symmetricgroup}, and arguing as in \Cref{section:char_values_and_conjugacy} gives the claim.
An analogous argument shows that for the alternating group \(A_n\), the values on cycles (of odd length) and values on products of two cycles of even length -- a set of size \(O(n^2)\) -- suffice to determine a character.
For the double covers \(\cover{A}_n\) and \(\cover{S}_n\), it can be shown that class sums of preimages of cycles of odd length together with all non-split even classes generate enough of the centre of the group algebra to determine characters up to associates or conjugacy; since \spin characters vanish on non-split classes, this implies that values on one of each of the split classes of cycles of odd length -- a set of size \(\ceil{\frac{n}{2}}\) -- suffice to determine a \spin character up to associates or conjugacy.

A related endeavour of Chow and Paulhus \cite{chowpaulhus2021algorithm} gives an algorithm to determine a character of the symmetric group from its values on a set of size \(O(n)\) -- although they require more values than just those on the cycles, they explicitly construct the indexing partition from these character values without requiring knowledge of the full character table.

\subsection{Remarks on spin characters modulo 2}
\label{subsection:spin_equal_to_non-spin}

Although \Cref{thm:maintheorem:double_cover_of_S_n,thm:maintheorem:double_cover_of_A_n} rule out two spin characters agreeing on \(2'\)-classes (except for the specified associate or conjugate pairs), there do exist spin characters which agree with a non-spin character on the \(2'\)-classes.
That is, although there are no repeated rows in the \(2\)-modular decomposition matrices of \(\cover{S}_n\) and \(\cover{A}_n\) with both rows labelled by spin characters, there are repeated rows with one row labelled by a spin character and the other by a non-spin character.
For example, when \(n=5\) the spin character \(\negchar{(4,1)}\) agrees with the non-spin character \(\chi^{(3,1^2)}\) on \(2'\)-classes \cite[p.~2]{ATLAS}, and the corresponding rows of the decomposition matrix are equal \cite[p.~885]{fayers2018irredchar2spinreps}.
(There are also repeated rows in the \(2\)-modular decomposition matrix of \(\cover{S}_n\) labelled by conjugate pairs of partitions, i.e. the repeated rows in the decomposition matrix of \(S_n\) classified by \cite[Theorem~1.1.1(ii)]{wildon2008distinctrows}.)

In fact, there exist infinitely many pairs of spin and non-spin characters which agree on \(2'\)-classes, and moreover whose affording representations are isomorphic after reduction modulo \(2\).
(Here the \(2\)-modular reduction of a representation of \(\cover{S}_n\) is viewed as a representation of \(S_n\), as is possible because the image of the central element of order \(2\) is trivial modulo \(2\)).
Indeed, consider \(\la = (k, k-1, \ldots, 1)\) with \(k \geq 2\), the \(2\)-core partition of \(n = \frac12 k(k+1)\).
Then \(\chi^\la\) is the only ordinary irreducible character of \(S_n\) lying in its own \(2\)-block, and hence its \(2\)-modular reduction remains irreducible and is the only irreducible \(2\)-modular character in that block.
It is known that, for \(\mu\) a strict partition, the spin character \(\negchar{\mu}\) lies in the same \(2\)-block as \(\chi^{\mathrm{dblreg}(\mu)}\), where \(\mathrm{dblreg}(\mu)\) denotes the \emph{2-regularisation} of the \emph{double} of \(\mu\), and furthermore \(\chi^{\mathrm{dblreg}(\mu)}\) occurs as a composition factor of \(\negchar{\mu}\) with multiplicity \(2^{\floor{m_0(\mu)/2}}\) where \(m_0(\mu)\) denotes the number of even parts of \(\mu\) \cite[Theorem~5.2]{bessenrodt1997coveringgroups2blocks}.
Pick \(\mu =  (2k-1, 2k-5, \ldots, 7, 3)\) or \(\mu = (2k-1, 2k-5, \ldots, 5, 1)\); it is easily verified that \(\mathrm{dblreg}(\mu) = \la\) and \(2^{\floor{m_0(\mu)/2}} =1 \), and hence the \(2\)-modular reductions of the representations affording \(\negchar{\mu}\) and \(\chi^\la\) are isomorphic.
Similar reasoning establishes isomorphisms between spin and non-spin representations in the \(2\)-block of weight~\(1\).
Nevertheless, not all such coincidences arise in this way (such as the example when \(n=5\) in the previous paragraph).

\section{(Spin) characters of symmetric and alternating groups}
\label{section:characters_of_A_n_and_covers}

\subsection{Symmetric and alternating groups and their characters}
\label{subsection:chars_of_S_n_and_A_n}

Here we state well-known results on the character theory of the symmetric and alternating groups; for a complete account, see, for example, \cite[Chapter~2]{jameskerber1984reptheory}.

The ordinary irreducible characters of \(S_n\) are indexed by partitions of \(n\), and we denote the character corresponding to \(\lambda\) as \(\chi^\lambda\).
If \(\lambda \neq \lambda'\), then \(\chi^\lambda\downto{A_n} = \chi^{\lambda'}\downto{A_n}\) is an irreducible character of \(A_n\).
If \(\lambda = \lambda'\),  then \(\chi^\lambda\downto{A_n}\) splits into a conjugate pair of irreducible characters of \(A_n\) denoted \(\chi^{\lambda^\pm}\) (the signs can be assigned arbitrarily).
Moreover, we obtain a complete irredundant set of ordinary irreducible characters of \(A_n\) in this way.

We denote by \(\Ccl_\la\) the \(S_n\)-conjugacy class of permutations of cycle type \(\la\).
The \(S_n\)-conjugacy classes which split in \(A_n\) are precisely those of cycle type with distinct parts all of which are odd; for \(\la\) with distinct parts all odd, we denote the resulting \(A_n\)-conjugacy classes by \(\Ccl_\la^\pm\) (the signs can be assigned arbitrarily).

A \emph{principal hook length} of a partition refers to a hook length of a box on the main diagonal in the Young diagram of the partition.
For \(\la\) self-conjugate, the principal hook lengths of \(\la\) are distinct odd integers; write \(\diag(\la)\) for the partition whose parts are the principal hook lengths of \(\la\).
The pair of characters of \(A_n\) labelled by a self-conjugate partition \(\la\) differ only on the split class of cycle type \(\diag(\la)\), from which we deduce the infinite family of characters agreeing on \(\el'\)-classes stated in \Cref{thm:maintheorem:A_n}.

\subsection{Double covers and spin characters}
\label{subsection:chars_of_double_covers}

Here we state elementary results on the projective representation theory of the symmetric and alternating groups; for a complete account, see, for example, \cite{hoffmanhumphreys1992projreps}.

A projective representation of a group \(G\) is a group homomorphism \(G \to \PGL(V)\) for some \(V\).
There is a correspondence between projective representations of \(G\) and linear representations of a central extension of \(G\) by its Schur multiplier.
When \(n \geq 4\) the Schur multiplier of the alternating and symmetric groups are cyclic of order \(2\) (except for \(n\in\set{6,7}\) for the alternating group), so it suffices to consider the linear representations of their double covers.

We use the double cover \(\cover{S}_n\) described in \cite[pp.~18--19]{hoffmanhumphreys1992projreps}, which has generators \(z, t_1, \ldots, t_{n-1}\) subject to the relations
\begin{gather*}
    \flexbox[c]{z^2 = 1;}{(t_jt_{j+1})^3 = z;} \quad\qquad
    \flexbox[c]{t_j^2 = z;}{(t_jt_{j+1})^3 = z;} \quad\qquad
    (t_jt_{j+1})^3 = z;  \\
    t_j t_k = z t_k t_j \,\text{ for \(\abs{j-k}>1\)}.
\end{gather*}
There is a projection map \(\projmap \colon \cover{S}_n \to S_n\) with kernel \(\set{1, z}\).
Define the sign map \(\cover{S}_n \to \set{\pm 1}\) by composition of the usual sign map with \(\projmap\); its kernel is \(\cover{A}_n\), a double cover of the alternating group.

The image of \(z\) under a representation of \(\cover{S}_n\) or \(\cover{A}_n\) is \(\pm I\); if it is \(I\), then the representation corresponds to a linear representation of \(S_n\) or \(A_n\); if it is \(-I\), then the representation is called \emph{spin} and it corresponds to a projective representation of \(S_n\) or \(A_n\) which is not linear.
Since the linear representations of \(S_n\) and \(A_n\) are already described, we are interested in the \spin representations of \(\cover{S}_n\) and \(\cover{A}_n\).

There is another choice of double cover of \(S_n\) (though not for \(A_n\) -- both covers of \(S_n\) yield isomorphic covers of \(A_n\)) \cite[pp,~22-23]{hoffmanhumphreys1992projreps}.
However, the \spin characters of one double cover can be obtained from the other by multiplying the values on \(\cover{S}_n \setminus \cover{A}_n\) by \(i\) \cite[\S6.7]{ATLAS}, 
and hence the choice of double cover makes no difference to the determination of the character by its values.
% also: \spin characters of \(\cover{S}_n\) are nonzero on \(\cover{S}_n \setminus \cover{A}_n\) only on the permutation of cycle type labelling the character anyway

For each strict partition (that is, partition with distinct parts) \(\la\) of \(n\), there is a \spin ordinary irreducible character \(\negchar{\la}\) of \(\cover{S}_n\).
Each character of \(\cover{S}_n\) has an \emph{associate}, denoted \(\blank^\ass\), which takes values of opposite sign on odd elements.
We call a partition \emph{even} or \emph{odd} according to whether permutations of that cycle type are even or odd.
If \(\la\) is odd, then \(\negchar{\la} \neq \negchar{\la}^\ass\), and \(\negchar{\la}\downto{\cover{A}_n} = \negchar{\la}^\ass\downto{\cover{A}_n}\) is irreducible and self-conjugate.
If \(\la\) is even, then \(\negchar{\la} = \negchar{\la}^\ass\), and \(\negchar{\la}\downto{\cover{A}_n}\) splits into a conjugate pair of irreducible characters of \(\cover{A}_n\) denoted \(\negchar{\la}^\pm\) (the signs can be assigned arbitrarily).
We obtain complete irredundant sets of ordinary irreducible \spin characters of \(\cover{S}_n\) and \(\cover{A}_n\) in this way.

Write \(\coverCcl\) for the preimage in \(\cover{S}_n\) of an \(S_n\)-conjugacy class \(\Ccl\).
Say elements of \(\coverCcl_\la\) are of cycle type \(\la\).
The set \(\coverCcl_\la\) is either itself an \(\cover{S}_n\)-conjugacy class, or splits into two \(\cover{S}_n\)-conjugacy class.
Splitting occurs if and only if \(\la\) has all parts odd or \(\la\) is odd with all parts distinct; in these cases, write \(\coverCcl_\la^\uno\) and \(\coverCcl_\la^\zed\) for the resulting \(\cover{S}_n\)-conjugacy classes (the labels can be assigned arbitrarily).
The \(\cover{S}_n\)-conjugacy classes which split in \(\cover{A}_n\) are precisely those of cycle type with all parts odd or all parts distinct (which includes all classes which split from \(S_n\) to \(\cover{S}_n\)); if \(\Ccl\) is such a \(\cover{S}_n\)-conjugacy class, we denote the resulting \(\cover{A}_n\)-conjugacy classes \(\Ccl^\pm\) (the signs can be assigned arbitrarily).

An associate pair of irreducible \spin characters of \(\cover{S}_n\), or a conjugate pair of irreducible \spin characters of \(\cover{A}_n\), differ only on the split class of cycle type equal to the labelling partition \cite[Theorem~8.7]{hoffmanhumphreys1992projreps}.
From this fact we deduce the infinite families of characters agreeing on \(\el'\)-classes stated in \Cref{thm:maintheorem:double_cover_of_S_n,thm:maintheorem:double_cover_of_A_n}.

An irreducible \spin character of \(\cover{S}_n\) or \(\cover{A}_n\) necessarily vanishes except on the classes of cycle type with all parts odd or of cycle type equal to the partition which labels the character %(with the latter exception occurring for \(\cover{S}_n\) only when the cycle type is odd) 
\cite[Theorem~8.7]{hoffmanhumphreys1992projreps}.
In particular, if \(\el\) is even, a pair of irreducible \spin characters of \(\cover{S}_n\) and \(\cover{A}_n\) which agree on \(\el'\)-classes can differ only on the classes of cycle type the partitions labelling the characters;
if the characters are not conjugate or associate (that is, not labelled by the same partition), then on each of these classes at least one of the pair vanishes,
% and hence the pair agree everywhere that they are both nonzero,
contradicting their orthogonality.
This yields the cases of \Cref{thm:maintheorem:double_cover_of_S_n,thm:maintheorem:double_cover_of_A_n} where \(\el\) is even.

\section{Character values and conjugacy of characters}
\label{section:char_values_and_conjugacy}

\newcommand{\U}{\mathcal{U}}
\newcommand{\M}{\mathcal{M}}
\newcommand{\Dcl}{\mathcal{D}}
\newcommand{\Ecl}{\mathcal{E}}

Let \(G\) be a finite group.
Given a conjugacy class \(\Ccl\), let \(s_\Ccl \in Z(\Q G)\) denote the class sum, and given a character \(\chi\) we write \(\chi(\Ccl)\) for the value \(\chi(g)\) for any \(g \in \Ccl\).
More generally, if \(\Dcl\) is a union of conjugacy classes of \(G\), write \(s_\Dcl = \sum_{\Ccl \subseteq \Dcl} s_{\Ccl}\) for the sum of the constituent class sums, and \(\chi(\Dcl) = \sum_{\Ccl \subseteq \Dcl} \chi(\Ccl)\) for the sum of character values on each constituent class.

Suppose \(G\) has a normal subgroup \(H\).
Recall that any \(G\)-conjugacy class is either disjoint from \(H\), or is a union of \(H\)-conjugacy classes of equal size.
The aim of this section is to prove a sufficient condition on the values of a pair of characters of \(H\) for the characters to be \(G\)-conjugate (\Cref{prop:distinguishing_up_to_conjugacy}).

\subsection{Central characters}

The following is the key lemma which allows us to deduce further relations between character values from a given set.
This argument was used by Wildon in \cite{wildon2008distinctrows}, though he considers only the case where the unions are single classes.

\begin{lemma}
\label{lemma:central_character_argument}
Let \(\chi\) and \(\psi\) be irreducible characters of \(G\) of equal degree.
Let \(\U\) be a set of unions of \(G\)-conjugacy classes, such that all classes in a given union are of the same size.
Suppose that \(\chi(\Dcl) = \psi(\Dcl)\) for all \(\Dcl \in \U\).
If \(\Ecl\) is a union of \(G\)-conjugacy classes of equal size such that \(s_{\Ecl}\) lies in the algebra generated by \(\setbuild{ s_\Dcl}{\Dcl \in \U}\),
then \(\chi(\Ecl) = \psi(\Ecl)\).
\end{lemma}

To prove this lemma, we use central characters.
A \emph{central character} of a group \(G\) is a \(\Q\)-algebra homomorphism \(Z(\Q G) \to \C\) (this is sometimes defined with other fields in place of \(\Q\) or \(\C\), but for our purposes this suffices).
Given an irreducible character \(\chi\) of \(G\), there is a corresponding central character \(\omega_\chi\) which sends an element \(\alpha \in Z(\Q G)\) to the scalar by which \(\alpha\) acts on the representation afforded by \(\chi\); thus on a conjugacy class sum \(s_\Ccl\) the central character \(\omega_\chi\) is defined by
\[
    \omega_\chi(s_\Ccl) = \frac{\abs{\Ccl}\chi(\Ccl)}{\chi(1)}.
\]

\begin{proof}[Proof of \Cref{lemma:central_character_argument}]
Write \(d\) for the common value \(\chi(1) = \psi(1)\).
Let \(\omega_\chi\), \(\omega_\psi\) denote the central characters corresponding to \(\chi\) and \(\psi\) respectively.
If \(\Dcl \in \U\) is a union of conjugacy classes of size \(m\), we have
\[
\omega_\chi (s_{\Dcl})
    = \frac{m}{d} \chi(\Dcl)
    = \frac{m}{d} \psi(\Dcl)
    = \omega_\psi ( s_{\Dcl} ),
\]
so \(\omega_\chi\) and \(\omega_\psi\) agree on \(\setbuild{s_{\Dcl}}{\Dcl \in \U}\) and hence on the algebra that set generates.
Suppose \(\Ecl\) is a union of \(G\)-conjugacy class of equal size \(r\) such that \(s_{\Ecl}\) lies in that algebra.
Then
\[
\chi(\Ecl)
    = \frac{d}{r} \omega_\chi (  s_{\Ecl} )
    = \frac{d}{r}  \omega_\psi( s_{\Ecl} )
    = \psi(\Ecl). \qedhere
\]
% as required.
\end{proof}

\subsection{Distinguishing characters up to conjugates}

\begin{lemma}
\label{lemma:char_value_on_split_class}
Let \(\chi\) be a character of \(H\).
\begin{enumerate}[(i)]
    \item\label{item:value_as_sum_over_classes}
Let \(\Ccl\) be a \(G\)-conjugacy class contained in \(H\),
and let \(k\) be the number of \(H\)-conjugacy classes into which \(\Ccl\) splits.
Then
\[
    \chi\upto{G}(\Ccl) = \frac{\abs{G : H}}{k}  \chi(\Ccl).
\]
% where \(k\) is the number of \(H\)-conjugacy classes into which \(\Ccl\) splits.
    \item\label{item:value_as_sum_over_conjugates}
Let \(\stab_G(\chi)\) denote the stabiliser of \(\chi\) under the conjugation action of \(G\), and let \(\sim\) denote the equivalence relation of \(G\)-conjugacy of characters. Then
\[
    \chi\upto{G}\downto{H} = \abs{\stab_G(\chi) : H}  \sum_{\psi \sim \chi} \psi.
\]
\end{enumerate}
\end{lemma}

\begin{proof}
The expression \(\chi\upto{G}(g) = \frac{1}{\abs{H}} \sum_{x \in G} \chi(x^{-1}gx)\) for \(g \in H\) yields both parts routinely
(for the first part, break up the sum over cosets of the centraliser of \(g\) in \(G\) and use that the \(H\)-conjugacy classes into which \(\Ccl\) splits are of equal size \(\abs{\Ccl}/k\); for the second part, break up the sum over cosets of the stabiliser of \(\chi\) in \(G\)).
\end{proof}

\begin{lemma}
\label{lemma:equivalence_to_G-conjugacy}
Let \(\chi\), \(\psi\) be irreducible characters of \(H\).
The following are equivalent:
\begin{enumerate}[(i)]
    \item\label{item:sum_of_char_values}
\(\chi(\Ccl) = \psi(\Ccl)\) for every \(G\)-conjugacy class \(\Ccl\) contained in \(H\);
    \item\label{item:induced_chars}
\(\chi\upto{G} = \psi\upto{G}\);
    \item\label{item:G-conjugate}
\(\chi\) and \(\psi\) are \(G\)-conjugate.
\end{enumerate}
\end{lemma}

\begin{proof}
Note characters induced from a normal subgroup are zero off that subgroup.
Then the equivalence of \ref{item:sum_of_char_values} and \ref{item:induced_chars} follows from \Cref{lemma:char_value_on_split_class}\ref{item:value_as_sum_over_classes}, and the equivalence of \ref{item:induced_chars} and \ref{item:G-conjugate} follows from \Cref{lemma:char_value_on_split_class}\ref{item:value_as_sum_over_conjugates}.
\end{proof}

\begin{proposition}
\label{prop:distinguishing_up_to_conjugacy}
Let \(\U\) be a set of unions of \(H\)-conjugacy classes, such that all classes in a given union are of the same size.
Suppose the subalgebra of \(Z(\Q H)\) generated by \(\setbuild{s_{\Dcl}}{\Dcl \in \U}\) contains all the sums of \(G\)-conjugacy classes contained in \(H\).
If \(\chi\) and \(\psi\) are irreducible characters of \(H\) of equal degree such that \( \chi(\Dcl) =  \psi(\Dcl)\) for all \(\Dcl \in \U\), then \(\chi\) and \(\psi\) are \(G\)-conjugate.
\end{proposition}

\begin{proof}
Combine \Cref{lemma:central_character_argument,lemma:equivalence_to_G-conjugacy}.
\end{proof}

\section{Generation of the centre of the group algebra}
\label{section:class_sum_arithmetic}

The aim of this section is to prove, working in \(Z(\Q A_n)\), that the even \(S_n\)-conjugacy \(\el'\)-class sums generate all the even \(S_n\)-conjugacy class sums, and likewise with \(\cover{S}_n\) and \(\cover{A}_n\) in place of \(S_n\) and \(A_n\) (requiring \(\el\) odd in this case).

For any group \(G\), the centre of the group algebra \(Z(\Q G)\) has linear basis the set of conjugacy class sums.
Furthermore, if \(H\) is a normal subgroup of \(G\) and \(\Ccl_1\), \(\Ccl_2\) are \(G\)-conjugacy classes in \(H\), then the product \(s_{\Ccl_1} s_{\Ccl_2}\) can be written uniquely as a positive integral linear combination of sums of \(G\)-conjugacy classes in \(H\).

\subsection{Centre of the group algebra of the alternating group}

We will induct on the following statistic.

\begin{definition}\
The \emph{support} of a partition \(\la\), denoted \(\support{\la}\), is the number of non-fixed points of a permutation of cycle type \(\la\).
\end{definition}

That is, the support of \(\la\) is the sum of the parts of \(\la\) strictly greater than \(1\).

Throughout this section, for convenience we omit writing the parts of a partition which are equal to \(1\), and allow a partition of any integer \(m \leq n\) to be viewed as a partition of \(n\) by appending \(1\)s as necessary.
Again for convenience, we do not necessarily write the parts of a partition in decreasing order.
These abuses of notation are acceptable here as they do not alter which conjugacy class of \(S_n\) is determined via cycle type.

Write \(s_\la = s_{\Ccl_\la}\) for the sum of the class \(\Ccl_\la\) of permutations of cycle type \(\la\).
Given partitions \(\la, \mu, \nu\), we say \(\nu\) is \emph{involved in} the product \(s_\la s_\mu\) if \(s_{\nu}\) appears with nonzero (hence positive integer) coefficient in the product (when written with respect to the \(S_n\)-conjugacy class basis).
This occurs if and only if there exist permutations \(\sigma, \tau, \rho\) of cycle types \(\la, \mu, \nu\) such that \(\sigma\tau = \rho\).

Given partitions \(\la\) and \(\mu\),
let \(\la \sqcup \mu\) denote the partition consisting of all the parts of \(\la\) and \(\mu\) (parts strictly greater than \(1\) only, arranged in decreasing order).
The following simple fact concerning multiplication of class sums was used by Kramer to show that the set \(\setbuild{s_{(i)}}{1 \leq i \leq n}\) generates \(Z(\Q S_n)\) \cite{kramer1966symmetricgroup}.

\begin{lemma}
\label{lemma:summand_of_greatest_support}
Let \(\la, \mu\) be partitions such that \(\support{\la} + \support{\mu} \leq n\).
The unique partition of support greater than or equal to \(\support{\la} + \support{\mu}\) which is involved in the product \(s_\la s_\mu\) is \(\la \sqcup \mu\).
\end{lemma}

\begin{proof}
Suppose \(\sigma,\tau\) are permutations of cycle type \(\la,\mu\) respectively.
In order for the product \(\sigma\tau\) to have at least \(\support{\la} + \support{\mu}\) non-fixed points, the non-fixed points of \(\sigma\) and \(\tau\) must be disjoint, in which case \(\sigma\) and \(\tau\) commute and the product \(\sigma\tau\) has cycle type \(\la\sqcup\mu\).
Indeed such permutations exist given \(\support{\la} + \support{\mu} \leq n\).
\end{proof}

We require a stronger version of \Cref{lemma:summand_of_greatest_support} in the case of multiplying by a cycle.

\begin{lemma}
\label{lemma:multiplying_cycle_with_partition}
Let \(\la\) be a partition with \(m\) parts (all strictly greater than \(1\)), and let \(r > 1\) be a positive integer such that \(\support{\la} + r \leq n\).
\begin{enumerate}[(i)]
\item
    The unique partition of support \(\support{\la} + r\) involved in the product \(s_{\la} s_{(r)}\) is \(\la \sqcup (r)\).
\item
    The partitions of support \(\support{\la} + r -1\) involved in the product \(s_{\la} s_{(r)}\) are (up to reordering parts) precisely those of the form \((\lambda_1, \ldots, \lambda_{i-1}, \lambda_i + r - 1, \lambda_{i+1}, \ldots, \lambda_{m})\) for some \(1 \leq i \leq m\).
\item
    A partition of support \(\support{\la} + r -2\) involved in the product \(s_{\la} s_{(r)}\) either has \(m+1\) parts strictly greater than \(1\), or is (up to reordering parts) of the form \((\lambda_1, \ldots, \lambda_{i-1}, \lambda_i + \la_j + r - 2, \lambda_{i+1}, \ldots, \la_{j-1}, \la_{j+1}, \ldots, \lambda_{m})\) for some \(1 \leq i < j \leq m\).
\end{enumerate}
\end{lemma}

\begin{proof}
The first part is a special case of \Cref{lemma:summand_of_greatest_support}.
For the second part, observe that the product of two non-identity permutations has exactly one fewer non-fixed point than the sum of their numbers of non-fixed points if and only if the two permutations have exactly one non-fixed point in common.
The product of a \(\la_i\)-cycle and \(r\)-cycle with exactly one non-fixed point in common is a \((\lambda_i+r-1)\)-cycle, yielding the specified form.

For the third part, observe that if the product of two non-identity permutations has exactly two fewer non-fixed points than the sum of their numbers of non-fixed points, then the two permutations have exactly two non-fixed point in common (though not conversely).
If an \(r\)-cycle has two non-fixed points in common with a \(\la_i\)-cycle, then either their product has cycle type \((a,b)\) with \(a+b = \la_i+r-2\) and we are in the case of having \(m+1\) parts strictly greater than \(1\), or the product is a \((\la_i+r-3)\)-cycle and we do not obtain a permutation of the specified support; if an \(r\)-cycle has one non-fixed point in common with a \(\la_i\)-cycle and one non-fixed point in common with a disjoint \(\la_j\)-cycle, the product is an \((\la_i+\la_j+r-2)\)-cycle and we are in the case of the form specified.
\end{proof}

\newcommand{\Zalg}{\mathcal{Z}}

\begin{theorem}
\label{thm:generation_theorem}
Let \(\el \geq 4\) and \(n \geq \el\).
The subalgebra of \(Z(\Q A_n)\) generated by the even \(S_n\)-conjugacy \(\el'\)-class sums contains all even \(S_n\)-conjugacy class sums.
\end{theorem}

\begin{proof}
Let \(\Zalg\) be the subalgebra generated by the even \(S_n\)-conjugacy \(\el'\)-class sums.
We aim to show \(s_\la \in \Zalg\) for all even partitions \(\la\) by induction on \(\support{\la}\) (primarily) and \(n-\la_1\) (secondarily).
The cases \(0 \leq \support{\la} \leq 3\) are trivial.
Suppose \(\support{\la} \geq 4\), and suppose we have shown that \(s_{\mu} \in \Zalg\) for all even partitions \(\mu\) of support strictly less than \(\la\) and for all partitions \(\mu\) of support equal to \(\la\) with \(\mu_1 > \la_1\).

The order of an element of \(S_n\) is the lowest common multiple of the numbers in its cycle type.
Thus if \(\el\) does not divide the lowest common multiple of the parts of \(\la\), the claim holds trivially.
Otherwise, let \(q \geq 3\) be a prime power dividing \(\el\), and note that \(q\) divides at least one part of \(\la\).

We consider three cases depending on the number and parity of the parts of \(\la\) (following our convention on omitting parts equal to \(1\)).

\proofpart{\(\la\) has at least three parts; or \(\la\) has exactly two parts, both odd}

We can write \(\la = \mu \sqcup \nu\) where \(\mu\) and \(\nu\) are nonempty even partitions with support strictly less than \(\support{\la}\).
Then \(s_\mu s_\nu \in \Zalg\) by the inductive hypothesis, and by \Cref{lemma:summand_of_greatest_support} the only partition of support at least \(\support{\la}\) involved in the product \(s_\mu s_\nu\) is \(\la\).
Using the inductive hypothesis again, all summands except \(s_\la\) lie in \(\Zalg\), and hence so does \(s_\la\).

\addvspace{\vsmallskipamount}

For the remaining two cases, choose a positive odd integer \(j\) such that  \(j \not\equiv 0\) and \(j \not\equiv 1\) (mod \(q\)), and such that \(j < a\) for every part \(a\) of \(\la\) which is divisible by \(q\).
If \(q>3\), then \(j=3\) suits; if \(q=3\), then \(j=5\) suits (noting that in the remaining two cases \(\la\) has either one odd part greater than \(4\), or two even parts).

\proofpart{\(\la\) has exactly one part, necessarily odd}

Write \(\la=(a)\), where \(a\) is odd and divisible by \(q\).
The partitions \((a{-}j{+}1)\) and \((j)\) are even and have support strictly less than \(\support{\la}\), so \(s_{(a{-}j{+}1)} s_{(j)} \in \Zalg\) by the inductive hypothesis.
By \Cref{lemma:multiplying_cycle_with_partition}, the only partitions involved in this product of support at least \(\support{\la}\) are \(\la\) and \((a{-}j{+}1,j)\).
Since \(a{-}j{+}1\) and \(j\) are not divisible by \(q\), we have \(s_{(a{-}j{+}1,j)} \in \Zalg\) trivially, and hence (using the inductive hypothesis again) we conclude \(s_{\la} \in \Zalg\).

\proofpart{\(\la\) has exactly two parts, both even}

Write \(\la = (a,b)\), where \(a\) and \(b\) are even and at least one is divisible by \(q\).
Suppose \(a \geq b\).
We consider four cases depending on the residues of \(a\) and \(b\) modulo \(q\).

\proofpart{\;\(\bullet\)\; \(a \equiv 0\), \(b \not\equiv 0,\!1\) (mod \(q\))}

Since \(a \not\equiv b\), we have \(a > b\).
The product \(s_{(b,b)} s_{(a{-}b{+}1)}\) lies in \(\Zalg\) by the inductive hypothesis.
Other than \(\la\), the only partition of support at least \(\support{\la}\) involved in the product is \((b, b, a{-}b{+}1)\).
This partition has no part divisible by \(q\), so \(s_{(b, b, a{-}b{+}1)} \in \Zalg\), and hence \(s_\la \in \Zalg\).

\proofpart{\;\(\bullet\)\; \(a \equiv 0\), \(b \equiv 1\) (mod \(q\))}

The product \(s_{(a{-}j{+}1, b)} s_{(j)}\) lies in \(\Zalg\) by the inductive hypothesis.
Other than \(\la\), the only partitions of support at least \(\support{\la}\) involved in the product are \((a{-}j{+}1, b{+}j{-}1)\) and \((a{-}j{+}1, b, j)\).
Neither of these partitions has a part divisible by \(q\), so the corresponding class sums lie in \(\Zalg\), and thus \(s_\la \in \Zalg\).

\proofpart{\;\(\bullet\)\; \(a \not\equiv 0\), \(b \equiv 0\) (mod \(q\))}

The product \(s_{(a, b{-}j{+}1)} s_{(j)}\) lies in \(\Zalg\) by the inductive hypothesis.
Other than \(\la\), the only partitions of support at least \(\support{\la}\) involved in this product are \((a, b{-}j{+}1, j)\) and \((a{+}j{-}1, b{-}j{+}1)\).
The former partition has no part divisible by \(q\), so \(s_{(a, b{-}j{+}1, j)} \in \Zalg\); the latter has support equal to that of \(\la\) but greater first part, so \(s_{(a{+}j{-}1, b{-}j{+}1)} \in \Zalg\) by the inductive hypothesis.
Thus \(s_\la \in \Zalg\).

\proofpart{\;\(\bullet\)\; \(a \equiv b \equiv 0\) (mod \(q\))}

As in the previous case, the product \(s_{(a, b{-}j{+}1)} s_{(j)}\) lies in \(\Zalg\) by the inductive hypothesis, and other than \(\la\) the only partitions of support at least \(\support{\la}\) involved in the product are \((a{+}j{-}1, b{-}j{+}1)\) and \((a, b{-}j{+}1, j)\).
The former has no part divisible by \(q\), so to deduce \(s_\la \in \Zalg\) it suffices to show \(s_{(a, b{-}j{+}1, j)} \in \Zalg\).

The product \(s_{(b{-}j{+}1, b{-}j{+}1, j)} s_{(a{-}b{+}j)}\) lies in \(\Zalg\) by the inductive hypothesis.
Other than the partition of interest \((a, b{-}j{+}1, j)\), by \Cref{lemma:multiplying_cycle_with_partition} the only partitions of support at least \(\support{\la}\) involved in the product are
\begin{align*}
\begin{gathered}
    (b{-}j{+}1,\, b{-}j{+}1,\, j,\, a{-}b{+}j), \\
    (b{-}j{+}1,\, b{-}j{+}1,\, a{-}b{+}2j{-}1),
\end{gathered}
&&
\begin{gathered}
    (a{+}b{-}j,\, j), \\
    (a{+}j{-}1,\, b{-}j{+}1),
\end{gathered}
\end{align*}
or are of support equal to \(\support{\la}\) and have three parts greater than \(1\).
Partitions of support \(\support{\la}\) with three (or more) parts greater than \(1\) have already been shown to lie in \(\Zalg\) during the first case of the inductive step, while three of the four listed partitions have no part divisible by \(q\).
It now suffices to show \( s_{(b{-}j{+}1, b{-}j{+}1, a{-}b{+}2j{-}1)} \in \Zalg\).

The product \(s_{(b{-}j{+}1, b{-}j{+}1)} s_{(a{-}b{+}2j{-}1)}\) lies in \(\Zalg\) by the inductive hypothesis.
Other than \((b{-}j{+}1, b{-}j{+}1, a{-}b{+}2j{-}1)\), the only partition of support at least \(\support{\la}\) involved in this product is \((a{+}j{-}1, b{-}j{+}1)\).
This partition has no part divisible by \(q\), so its class sum lies in \(\Zalg\), and thus \(s_{(b{-}j{+}1, b{-}j{+}1, a{-}b{+}2j{-}1)} \in \Zalg\). %as required.
\end{proof}

\subsection{Centre of the group algebra of the double cover}

Working now with \(\cover{S}_n\) and \(\cover{A}_n\), the classes which split from \(S_n\) to \(\cover{S}_n\) require further attention.
Recall that \(\projmap \colon \cover{S}_n \to S_n\) denotes the projection map; that \(\coverCcl_\la\) denotes the preimage of \(\Ccl_\la\) under \(\projmap\); and that for \(\la\) even, \(\coverCcl_\la\) splits if and only if \(\la\) has all parts odd, in which case its components are denoted \(\coverCcl^\uno_\la\) and \(\coverCcl^\zed_\la\).
Let \(\cover{s}_\la\) denote the sum of \(\coverCcl_\la\), and if \(\coverCcl_\la\) splits let \(\cover{s}^\uno_\la\) and \(\cover{s}^\zed_\la\) denote the sums of the components. 

\begin{lemma}
\label{lemma:product_of_cycles_conjugate_to_nice_expression}
Let \(a,b \geq 2\), let \(\rho = (1\ 2\ \cdots\ {a{+}b{-}1}) \in S_n\), and let \(\sigma, \tau \in S_n\) be \(a\)- and \(b\)-cycles such that \(\sigma\tau = \rho\).
Then there exists \(k \in \Z\) such that \(\rho^k \sigma \rho^{-k} = (1\ 2\ \cdots\ a)\) and \(\rho^k \tau \rho^{-k} = (a\ \ a{+}1\ \cdots\ a{+}b{-}1)\).
\end{lemma}

\begin{proof}
All non-fixed points of the relevant permutations are in the range \(1\) to \(a{+}b{-}1\), so write integers modulo \(a{+}b{-}1\).
The permutations \(\sigma\) and \(\tau\) have a unique common non-fixed point; call it \(r\).
Then for \(1 \leq i < b\), we have that \(\tau^i(r)\) is fixed by \(\sigma\), and so \(\tau^i(r) = \sigma\tau^i(r) = \rho\tau^{i-1}(r) = \tau^{i-1}(r)+1\).
Thus 
\begin{gather*}
\tau = (r\ \ r{+}1\ \cdots\ r{+}b{-}1) = \rho^{r-a}(a\ \ a{+}1\ \cdots\ a{+}b{-}1)\rho^{a-r} \\
\intertext{and, since \(\sigma = \rho \tau^{-1}\), hence}
\sigma = \rho^{r-a} \big( \,\rho \,(a\ \ a{+}b{-}1\ \ a{+}b{-}2 \ \cdots\ a{+}1)\, \big) \rho^{a-r} = \rho^{r-a} (1\ 2\ \cdots\ a) \rho^{a-r}. \qedhere
\end{gather*}
\end{proof}

\begin{lemma}
\label{lemma:split_ccl_product}
Let \(\la, \mu, \nu\) be partitions with all parts odd such that \(\support{\nu} \geq \support{\la} + \support{\mu} - 1\).
All elements of \(\cover{S}_n\) of cycle type \(\nu\) appearing in the product \(\cover{s}_{\la}^\uno \cover{s}_{\mu}^\uno\) are conjugate in \(\cover{S}_n\).
\end{lemma}

\begin{proof}
Suppose \(g,g' \in \coverCcl^\uno_{\la}\) and \(h,h' \in \coverCcl^\uno_\mu\) are such that \(gh\) and \(g'h'\) are of cycle type \(\nu\).
We must show that \(gh\) is conjugate to \(g'h'\).
Since an element of a conjugacy class appears in a product of class sums if and only if all the elements of that class do (and since an element \(x\) of a split class is not conjugate to \(zx\)) we may assume furthermore that \(\projmap(gh)=\projmap(g'h')\) (that is, that \(g'h' \in \set{gh, zgh}\)).

Our strategy is to show that \(\projmap(g),\projmap(g')\) and \(\projmap(h),\projmap(h')\) are conjugate by the same conjugating element of \(S_n\).
Supposing this is done, then there exists \(x \in \cover{S}_n\) such that \(\projmap(x^{-1}gx) = \projmap(g')\) and \(\projmap(x^{-1}hx) = \projmap(h')\); since \(g,g'\) and \(h,h'\) are conjugate in \(\cover{S}_n\), this implies \(x^{-1}gx = g'\) and \(x^{-1}hx=h'\) and hence \(g'h' = x^{-1}ghx\) as required.
There are two cases depending on the support of \(\nu\), which determines the number of non-fixed points that \(\projmap(g)\) and \(\projmap(h)\) (and \(\projmap(g')\) and \(\projmap(h')\)) have in common.

If \(\support{\nu} = \support{\la} + \support{\mu}\), there are no non-fixed points in common.
Then the factorisations \(\projmap(g)\projmap(h)\) and \(\projmap(g')\projmap(h')\) are both decompositions of \(\projmap(gh) = \projmap(g'h')\) into disjoint cycles.
Such a decomposition is unique up to reordering, so \(\projmap(g)\) and \(\projmap(g')\) are conjugate by a permutation defined as follows: for each cycle in \(\projmap(g)\) but not \(\projmap(g')\) (hence in \(\projmap(h')\) but not \(\projmap(h)\)), choose a cycle of the same length in \(\projmap(g')\) but not \(\projmap(g)\) (hence in  \(\projmap(h)\) but not \(\projmap(h')\)), and swap the entries of these cycles (respecting the cyclic ordering).
Furthermore, \(\projmap(h)\) and \(\projmap(h')\) are conjugate by the inverse of this permutation.
But the permutation described is a product of disjoint transpositions, so is its own inverse.

If \(\support{\nu} = \support{\la} + \support{\mu} - 1\), there is exactly one non-fixed point in common.
Let \(a,b\) be the lengths of the cycles in \(\projmap(g), \projmap(h)\) containing this non-fixed point.
Without loss of generality, suppose these cycles are \((1\ 2\ \cdots a)\) and \((a\ \ a{+}1\ \cdots \ a{+}b{-}1)\), whose product is \(\rho = (1\ 2\ \cdots\ a{+}b{-}1)\).
Let \(\sigma\) and \(\tau\) be the \(a\)- and \(b\)-cycles in \(\projmap(g')\) and \(\projmap(h')\) whose product is an \((a{+}b{-}1)\)-cycle \(\pi\).
Since \(\projmap(gh) = \projmap(g'h')\),
either \(\pi = \rho\) or \(\pi\) and \(\rho\) are disjoint.
If \(\pi = \rho\), apply \Cref{lemma:product_of_cycles_conjugate_to_nice_expression} to find a permutation by which \(\sigma\) and \(\tau\) are conjugate to \((1\ 2\ \cdots\ a)\) and \((a\ \ a{+}1\ \cdots \ a{+}b{-}1)\), and which fixes integers greater than \(a{+}b{-}1\).
If \(\pi\) and \(\rho\) are disjoint, consider the product of disjoint transpositions swapping the entries of \(\pi\) and \(\rho\) (respecting the cyclic ordering).
In either case, \(\projmap(g')\) and \(\projmap(h')\) are conjugate by the chosen permutation to \(\projmap(g)\) and \(\projmap(h)\) respectively, as required.
\end{proof}

\begin{theorem}
\label{thm:double_cover_of_S_n_generation_theorem}
Let \(\el \geq 5\) be odd and let \(n \geq \el\).
The subalgebra of \(Z(\Q \cover{A}_n)\) generated by the even \(\cover{S}_n\)-conjugacy \(\el'\)-class sums contains all even \(\cover{S}_n\)-conjugacy class sums.
\end{theorem}

\begin{proof}
The order of an element of \(\cover{S}_n\) is either the same as or twice that of its image in \(S_n\).
Thus for \(\el\) odd, an \(\cover{S}_n\)-class is \(\el'\) if and only if its image in \(S_n\) is (but this can be false for \(\el\) even).

With this in mind, the proof of the theorem (with the restriction to \(\el\) odd) is essentially the same as that of \Cref{thm:generation_theorem}, replacing all \(S_n\)-conjugacy class sums with \(\cover{S}_n\)-conjugacy class sums.
The only addition is to use \Cref{lemma:split_ccl_product} in the steps where \(\la\) has all parts odd, in order to deduce that we obtain the split class sums.
\end{proof}

\section{Proofs of the main theorems}
\label{section:main_proofs}

We now deduce our main theorems.
The first main theorem is that irreducible characters of \(A_n\) are uniquely determined by their values on \(\el'\)-classes except for the pairs labelled by self-conjugate partitions with a principal hook length divisible by \(\el\).
The case \(\el=2\) is covered by \cite[Theorem~3.2.1]{wildon2008distinctrows}; the following proof deals with the case \(\el \geq 4\).

\begin{proof}[Proof of \Cref{thm:maintheorem:A_n}]
Suppose that \(\chi\) and \(\psi\) are irreducible characters of \(A_n\) which agree on the \(\el'\)-classes of \(A_n\).
Let \(\U\) be the set of even \(\el'\)-classes of \(S_n\).
By \Cref{thm:generation_theorem}, the algebra generated by \(\setbuild{s_\Dcl}{\Dcl \in \U}\) contains all even \(S_n\)-conjugacy class sums.
By hypothesis, \(\chi(\Dcl) = \psi(\Dcl)\) for all \(\Dcl \in \U\).
Thus by \Cref{prop:distinguishing_up_to_conjugacy}, \(\chi\) and \(\psi\) are \(S_n\)-conjugate.
The \(S_n\)-conjugate characters of \(A_n\) are precisely the pairs \((\chi^{\la^\plus}, \chi^{\la^\minus})\), where \(\la = \la'\), and these characters agree on \(\el'\)-classes if and only if \(\la\) has a principal hook length divisible by \(\el\) (which is necessarily odd).
\end{proof}

Another consequence of \Cref{thm:generation_theorem} is a sufficient condition in terms of character values for two partitions to be conjugate.
It is elementary that partition conjugacy can be determined by examining character values on \emph{all} even classes; it is sufficient to examine the (even) \(2'\)-classes by \cite[Corollary 2.2.4]{wildon2008distinctrows}.
We obtain the following analogue for integers at least \(4\).

\begin{corollary}
Let \(\el \geq 4\). %and \(n \geq \el\).
Let \(\lambda\) and \(\mu\) be distinct partitions of \(n\).
The characters \(\chi^\lambda\) and \(\chi^\mu\) agree on all even \(\el'\)-classes of \(S_n\) if and only if \(\la\) and \(\mu\) are conjugate.
\end{corollary}

\begin{proof}
Suppose \(\chi^\la\) and \(\chi^\mu\) agree on all even \(\el'\)-classes of \(S_n\).
Letting \(\U\) be the set of even \(\el'\)-classes of \(S_n\)
and using \Cref{lemma:central_character_argument} and \Cref{thm:generation_theorem}, we have that \(\chi^\la\) and \(\chi^\mu\) agree on all even classes of \(S_n\).
Irreducible characters of \(S_n\) with equal restriction to \(A_n\) are precisely those labelled by conjugate partitions, so \(\la\) and \(\mu\) are conjugate.
The converse is elementary.
\end{proof}

Moving on to \(\cover{S}_n\) and \(\cover{A}_n\), we deduce that irreducible \spin characters are determined by their values on \(\el'\)-classes, with the exception of associate or conjugate pairs labelled by a partition with a part divisible by \(\el\).
The case of \(\el\) even is immediate, as noted in \Cref{section:characters_of_A_n_and_covers}.
The following proofs hold for odd \(\el \geq 5\).

\begin{proof}[Proof of \Cref{thm:maintheorem:double_cover_of_S_n}]
Suppose that \(\chi\) and \(\psi\) are irreducible \spin characters of \(\cover{S}_n\) which agree on the \(\el'\)-classes of \(\cover{S}_n\).
Let \(\U\) be the set of even \(\el'\)-classes of \(\cover{S}_n\).
By \Cref{thm:double_cover_of_S_n_generation_theorem}, the algebra generated by \(\setbuild{s_\Dcl}{\Dcl \in \U}\) contains all even \(\cover{S}_n\)-conjugacy class sums.
By hypothesis, \(\chi(\Dcl) = \psi(\Dcl)\) for all \(\Dcl \in \U\).
Thus by \Cref{lemma:central_character_argument}, \(\chi\) and \(\psi\) agree on all even \(\cover{S}_n\)-conjugacy classes; that is, \(\chi\downto{\cover{A}_n} = \psi\downto{\cover{A}_n}\).
Spin characters of \(\cover{S}_n\) with equal restriction to \(\cover{A}_n\) are precisely the associate pairs \((\negchar{\la}, \negchar{\la}^\ass)\), where \(\la\) is an odd partition with distinct parts, and these characters agree on \(\el'\)-classes if and only if \(\la\) has a part divisible by \(\el\).
\end{proof}

\newcommand{\mylen}{\widthof{\(\la\) even with no part divisible by \(\el\)}}

\begin{proof}[Proof of \Cref{thm:maintheorem:double_cover_of_A_n}]
Suppose that \(\chi\) and \(\psi\) are irreducible \spin characters of \(\cover{A}_n\) which agree on the \(\el'\)-classes of \(\cover{A}_n\).
Let \(\U\) be the set of even \(\el'\)-classes of \(\cover{S}_n\).
By \Cref{thm:double_cover_of_S_n_generation_theorem}, the algebra generated by \(\setbuild{s_\Dcl}{\Dcl \in \U}\) contains all even \(\cover{S}_n\)-conjugacy class sums.
By hypothesis, \(\chi(\Dcl) = \psi(\Dcl)\) for all \(\Dcl \in \U\).
Thus by \Cref{prop:distinguishing_up_to_conjugacy}, \(\chi\) and \(\psi\) are \(\cover{S}_n\)-conjugate.
The \(\cover{S}_n\)-conjugate characters of \(\cover{A}_n\) are precisely the pairs \((\negchar{\la}^\plus, \negchar{\la}^\minus)\), where \(\la\) is even, and these characters agree on \(\el'\)-classes if and only if \(\la\) has a part divisible by \(\el\).
\end{proof}

\section{Characters agreeing on \texorpdfstring{\(3'\)}{3'}-classes}
\label{section:3-indistinguishable}

In addition to the conjugate pairs identified in the introduction, two more infinite families of characters of \(A_n\) agreeing on \(3'\)-classes are given by the following theorem.
A computer search by Mark Wildon has shown that, other than the \(S_n\)-conjugate pairs labelled by partitions with a principal hook length divisible by \(3\), these are the only characters of \(A_n\) agreeing on \(3'\)-classes for \(n \leq 33\).

\begin{theorem}
\label{thm:3-indistinguishable}
Let \(n \geq 3\) and let \(\nu\) be a self-conjugate partition of \(n-3\).
\begin{enumerate}[(i)]
    \item
Suppose \(\nu\) is \(3\)-core.
Let \(\lambda\) be the partition of \(n\) obtained by adding a \(3\)-hook to the first row of \(\la\) (so that \(\la'\) is the partition obtained by adding a \(3\)-hook to the first column of \(\la\)),
and let \(\mu\) be the self-conjugate partition of \(n\) obtained by adding a principal \(3\)-hook to \(\nu\).
Then \(\chi^{\la}\downto{A_n}\), \(\chi^{\mu^\plus}\) and \(\chi^{\mu^\minus}\) agree on \(3'\)-classes
    \item
Suppose \(\nu\) has a unique \(3\)-hook (necessarily on the diagonal).
Let \(\la\), \(\la'\), \(\mu\) and \(\mu'\) be the four partitions of \(n\) which can be obtained from \(\nu\) by adding a \(3\)-hook.
Then \(\chi^\la\downto{A_n}\) and \(\chi^\mu\downto{A_n}\) agree on \(3'\)-classes.
\end{enumerate}
There is at most one \(\nu\) satisfying each of the hypotheses for each \(n\).
\end{theorem}

\begin{proof}
\cite[Theorem~21.7]{gdjames1978reptheorysymgroups} states that if \(\nu\) is a partition of \(n-r\), then the generalised character \(\sum_{\la} (-1)^i \chi^\la\) vanishes on all classes except those containing an \(r\)-cycle, where the sum is over all partitions \(\la\) obtained from \(\nu\) by adding an \(r\)-hook, and \(i\) is the leg length of the added hook.
In case (i), then, we have equality \(\chi^\la + \chi^{\la'} = \chi^\mu\) on \(3'\)-classes;
in case (ii), we have equality \(\chi^\la + \chi^{\la'} = \chi^\mu + \chi^{\mu'}\) on \(3'\)-classes
(the fact that partitions with exactly zero or one \(3\)-hooks have exactly three or four addable \(3\)-hooks is clear from interpreting partitions on the \(3\)-abacus).
Restricting to \(A_n\) gives the result.
\end{proof}

\begin{remark}
If \(\la\) and \(\mu\) are partitions as in \Cref{thm:3-indistinguishable} and are of \(3\)-weight \(1\) or \(2\), then the characters of \(A_n\) labelled by \(\la\) and \(\mu\) vanish on the same conjugacy classes.
This can be shown using \cite[Theorem~3.9]{gllv2022SBCs} -- which extends \cite[Theorem~21.7]{gdjames1978reptheorysymgroups} to give the values of the considered generalised character on all classes -- whilst observing that at least one of \(\la\) and \(\mu\) have a unique \(3\)-hook.

A conjecture of Christine Bessenrodt, reported by Chris Bowman \cite[Conjecture~7]{bowman2022owr}, asserts that these are the only non-conjugate pairs of characters of \(A_n\)  that have equal vanishing sets (with three exceptions:
\begin{itemize}[leftmargin=50pt]
\item[\(n=15\):]
\((8, 3, 2, 1^2)\) and \((5, 4^2, 1^2)\);
\item[\(n=16\):]
\((8, 3^2, 1^2)\) and \((5^2, 4, 1^2)\);
\item[\(n=23\):]
\((8, 5, 4^2, 2)\) and \((7, 6, 4^2, 2)\);
\end{itemize} 
these pairs have equal vanishing sets but have \(3\)-weight \(5\)).
Indeed, the partitions \(\lambda\) and \(\mu\) described in \Cref{thm:3-indistinguishable} have equal \(3\)-cores and (up to conjugacy) have \(3\)-quotients \(((1), \emptyset, \emptyset)\) and \((\emptyset, (1), \emptyset)\) in case (i), and \(((1), (r^r), \emptyset)\) and \((\emptyset, (r+1,r^{r-1}), \emptyset)\) for some \(r \geq 1\) in case (ii); this description, restricting in case (ii) to \(r=1\), coincides with that of \cite[Conjecture~7]{bowman2022owr}.

Thus in all known cases a non-conjugate pair which have equal vanishing sets also agree on \(3'\)-classes (including the three exceptional pairs above, which fall into case (ii) in \Cref{thm:3-indistinguishable}; in general the pairs in case (ii) having \(3\)-weight exceeding \(2\) do not have equal vanishing sets).
This suggests an intriguing connection between \(3'\)-classes and vanishing sets for characters of \(A_n\).
\end{remark}

Below we record the characters of \(\cover{S}_n\) and \(\cover{A}_n\) which agree on \(3'\)-classes for \(n \leq 14\), identified using the character tables in \cite{hoffmanhumphreys1992projreps}.

\begin{proposition}
\label{prop:n<=14}
\begin{enumerate}[(i), beginthm]
    \item 
The following sets of \spin characters of \(\cover{S}_n\) agree on \(3'\)-classes:
\begin{itemize}[leftmargin=50pt]
    \item[\emph{\(n=5\):}]
\(\negchar{5}\) and \(\negchar{3,2}\);
    \item[\emph{\(n=6\):}]
\(\negchar{6}\) and \(\negchar{3,2,1}\);
    \item[\emph{\(n=7\):}]
\(\negchar{6,1}\) and \(\negchar{4,3}\);
    \item[\emph{\(n=8\):}]
\(\negchar{7,1}\) and \(\negchar{4,3,1}\); \(\negchar{6,2}\) and \(\negchar{5,3}\);
    \item[\emph{\(n=10\):}]
\(\negchar{8,2}\) and \(\negchar{5,3,2}\); \(\negchar{8,2}\) and \(\negchar{5,3,2}\);
    \item[\emph{\(n=11\):}]
\(\negchar{7,3,1}\) and \(\negchar{6,4,1}\);
    \item[\emph{\(n=12\):}]
\(\negchar{9,3}\) and \(\negchar{6,3,2,1}\);
    \item[\emph{\(n=13\):}]
\(\negchar{9,3,1}\) and \(\negchar{6,4,3}\); \(\negchar{8,3,2}\) and \(\negchar{6,5,2}\);
    \item[\emph{\(n=14\):}]
\(\negchar{9,3,2}\), \(\negchar{8,3,2,1}\) and \(\negchar{6,5,3}\).
\end{itemize}
The only other sets of \spin characters of \(\cover{S}_n\) agreeing on \(3'\)-classes for \(n \leq 14\) are
the associate pairs labelled by odd partitions with a part divisible by \(3\),
and those obtained from the above by replacing characters with their associates.
    \item
The following sets of \spin characters of \(\cover{A}_n\) agree on \(3'\)-classes.
\begin{itemize}[leftmargin=50pt]
    \item[\emph{\(n=12\):}]
\(\negchar{5,4,3}\downto{\cover{A}_n}\), \(\negchar{9,3}^\pm\) and \(\negchar{6,3,2,1}^\pm\);
    \item[\emph{\(n=13\):}]
\(\negchar{7,3,2,1}\downto{\cover{A}_n}\), \(\negchar{9,3,1}^\pm\) and \(\negchar{6,4,3}^\pm\).
\end{itemize}
The only other sets of \spin characters of \(\cover{A}_n\) agreeing on \(3'\)-classes for \({n \leq 14}\) are the conjugate pairs labelled by even partitions with a part divisible by \(3\),
and those obtained from the list for \(\cover{S}_n\) by restriction.
\end{enumerate}
\end{proposition}

\section*{Acknowledgements}

The author is grateful to Mark Wildon for the computations mentioned in \Cref{section:3-indistinguishable}, to Chris Bowman for sharing his work on vanishing sets of characters of \(A_n\), and to anonymous referees for their feedback on earlier versions of this paper.

\newcommand\ATlabel{}\newcommand\AT[2]{ATLAS}\newcommand\MDlabel{}\newcommand\MD[2]{McDW22}\newcommand\MSlabel{}\newcommand\MS[2]{McSO22}


\begin{thebibliography}{GLLV22}

\bibitem[\AT85]{ATLAS}
\ATlabel{\initials{J.H.} Conway, \initials{R.T.} Curtis, \initials{S.P.}
  Norton, \initials{R.A.} Parker and \initials{R.A.} Wilson}.
\newblock {\em {\(\mathbb{ATLAS}\)} of finite groups}.
\newblock Oxford University Press, Eynsham, 1985.
\newblock With computational assistance from J. G. Thackray.

\bibitem[BO97]{bessenrodt1997coveringgroups2blocks}
Christine Bessenrodt and J\o rn~B. Olsson.
\newblock The {$2$}-blocks of the covering groups of the symmetric groups.
\newblock {\em Advances in Mathematics}, 129(2):261--300, 1997.

\bibitem[Bow22]{bowman2022owr}
Chris Bowman.
\newblock Tensor products, modular representations, and character vanishing.
\newblock In {\em Character Theory and Categorification (Report No.~39/2022)}, volume 19(3) of {\em Oberwolfach Reports}. Mathematisches Forschunginstitut Oberwolfach, 2022.
\newblock Available at \href{https://doi.org/10.14760/OWR-2022-39}{doi.org/10.14760/OWR-2022-39}.

\bibitem[CP21]{chowpaulhus2021algorithm}
Timothy~Y. Chow and Jennifer Paulhus.
\newblock Algorithmically distinguishing irreducible characters of the
  symmetric group.
\newblock {\em Electronic Journal of Combinatorics}, 28:P2.5, 2021.

\bibitem[Fay18]{fayers2018irredchar2spinreps}
Matthew Fayers.
\newblock Irreducible projective representations of the symmetric group which
  remain irreducible in characteristic 2.
\newblock {\em Proceedings of the London Mathematical Society},
  116(4):878--928, 2018.

\bibitem[GLLV22]{gllv2022SBCs}
Eugenio Giannelli, Stacey Law, Jason Long, and Carolina Vallejo.
\newblock Sylow branching coefficients and a conjecture of {M}alle and
  {N}avarro.
\newblock {\em Bulletin of the London Mathematical Society}, 54(2):552--567,
  2022.

\bibitem[Hag71]{hagis1971partitions}
Peter Hagis, Jr.
\newblock Partitions with a restriction on the multiplicity of the summands.
\newblock {\em Transactions of the American Mathematical Society},
  155(2):375--384, 1971.

\bibitem[HH92]{hoffmanhumphreys1992projreps}
{\initials{P.N.}}~Hoffman and {\initials{J.F.}}~Humphreys.
\newblock {\em Projective Representations of the Symmetric Groups: Q-Functions
  and Shifted Tableaux}.
\newblock Oxford Mathematical Monographs. Clarendon Press, 1992.

\bibitem[HR18]{hardyramanujan1918asymptotic}
{\initials{G.H.}}~Hardy and S.~Ramanujan.
\newblock Asymptotic formulae in combinatory analysis.
\newblock {\em Proceedings of the London Mathematical Society}, s2-17:75--115,
  1918.

\bibitem[Jam78]{gdjames1978reptheorysymgroups}
{\initials{G.D.}}~James.
\newblock {\em The Representation Theory of the Symmetric Groups}, volume 682
  of {\em Lecture Notes in Mathematics}.
\newblock Springer-Verlag, 1978.

\bibitem[JK84]{jameskerber1984reptheory}
G.D. James and A.~Kerber.
\newblock {\em The Representation Theory of the Symmetric Group}.
\newblock Encyclopedia of Mathematics and its Applications. Cambridge
  University Press, 1984.

\bibitem[Kra66]{kramer1966symmetricgroup}
P.~Kramer.
\newblock Factorization of projection operators for the symmetric group.
\newblock {\em Zeitschrift für Naturforschung A}, 21(5):657--658, 1966.

\bibitem[Web16]{webb2016finitegroupreptheory}
Peter Webb.
\newblock {\em A Course in Finite Group Representation Theory}, volume 161 of
  {\em Cambridge Studies in Advanced Mathematics}.
\newblock Cambridge University Press, 2016.

\bibitem[Wil08]{wildon2008distinctrows}
Mark Wildon.
\newblock Character values and decomposition matrices of symmetric groups.
\newblock {\em Journal of Algebra}, 319(8):3382--3397, 2008.

\end{thebibliography}
\end{document}